\documentclass[a4paper,10pt]{amsart}
\usepackage{enumerate, amsmath, amsfonts, amssymb, amsthm, mathtools, thmtools, wasysym, graphics, graphicx, xcolor, frcursive,xparse,comment,ytableau,bbm}
\usepackage[vcentermath]{youngtab}
\usepackage[all]{xy}
\usepackage{hyperref}

\usepackage{url, hypcap}
\hypersetup{colorlinks=true, citecolor=darkblue, linkcolor=darkblue}

\usepackage{tikz}
\usetikzlibrary{calc,through,backgrounds,shapes,matrix}
\usetikzlibrary{math}
\usepackage{xifthen}
\usepackage{xstring}

\definecolor{darkblue}{rgb}{0.0,0,0.7} 
\newcommand{\darkblue}{\color{darkblue}} 
\definecolor{darkred}{rgb}{0.7,0,0} 
\definecolor{lightgrey}{rgb}{0.7,0.7,0.7} 

\usepackage{graphicx}
\def\chk#1{#1^{\smash{\scalebox{.7}[1.4]{\rotatebox{90}{\guilsinglleft}}}}}

\def\m{m}

\newcommand{\op}[1]{\operatorname{#1}}

\newcommand{\Expt}[2]{\operatorname*{\mathbb{E}}\limits_{#1}(#2)}

\newcommand{\core}{{\sf core}}

\newcommand{\size}{{\sf size}}
\newcommand{\CC}{{\mathbb{C}}}

\newcommand{\V}{V}

\renewcommand{\mod}{\operatorname{mod}}

\newcommand{\gf}{\mathfrak{g}}
\newcommand{\hf}{\mathfrak{h}}

\newtheorem{theorem}{Theorem}[section]
\newtheorem{proposition}[theorem]{Proposition}
\newtheorem{corollary}[theorem]{Corollary}
\newtheorem{lemma}[theorem]{Lemma}

\theoremstyle{definition}

\newtheorem{example}[theorem]{Example}

\usepackage[nameinlink]{cleveref}
\usepackage[all]{xy}
\usepackage[T1]{fontenc}
\crefformat{footnote}{#2\footnotemark[#1]#3}
\crefformat{conjecture}{Conjecture~#2#1#3}

\usepackage[colorinlistoftodos]{todonotes}

\newcommand{\defn}[1]{\emph{\darkblue #1}}

\title[Winnie the Pooh]
  {Strange Expectations and the Winnie-the-Pooh Problem}

\author[M.~Thiel]{Marko Thiel}
\address[M.~Thiel]{Jane Street Capital, London, United Kingdom}
\email{thiel.marko@gmail.com}

\author[N.~Williams]{Nathan Williams}
\address[N.~Williams]{University of Texas at Dallas, USA}
\email{nathan.f.williams@gmail.com}

\date{\today}

\begin{document}

\begin{abstract}
Motivated by the study of simultaneous cores, we give three proofs (in varying levels of generality) that the expected norm of a weight in a highest weight representation $\V_\lambda$ of a complex simple Lie algebra
$\mathfrak{g}$ is $\frac{1}{h+1}(\lambda + 2\rho, \lambda)$.  First, we argue directly using the polynomial method and the Weyl character formula.  Second, we use the combinatorics of semistandard tableaux to obtain the result in type $A$.  Third, and most interestingly, we relate this problem to the ``Winnie-the-Pooh problem'' regarding orthogonal decompositions of Lie algebras; although this approach offers the most explanatory power, it applies only to Cartan types other than $A$ and $C$.  We conclude with computations of many combinatorial cumulants.
\end{abstract}


\maketitle

\section{Introduction}

The representation theory of complex simple Lie algebras is a classical source of algebraic combinatorics, intimately related to tableaux and plane partitions, symmetric functions and positivity questions, quantum groups and crystals, and the plactic monoid and RSK.

Having fixed a Cartan subalgebra $\hf$, the finite-dimensional irreducible representions of a complex simple Lie algebra $\gf$ are completely classified by the dominant weights $\lambda$ in its weight lattice $\Lambda \subset \hf^*$.  Using the symmetric bilinear form $\langle \cdot,\cdot \rangle$ on $\hf^*$ induced by the Killing form and writing $\rho$ for the half-sum of the positive roots $\Phi^+$, the Weyl dimension formula asserts that for $\lambda$ dominant, the dimension of the finite-dimensional irreducible representation $V_\lambda$
\begin{equation}
\mathrm{dim}(\V_\lambda) = \prod_{\alpha \in \Phi^+} \frac{\langle \alpha,\lambda+\rho \rangle}{\langle \alpha,\rho \rangle}.
\end{equation}

Motivated by the recent interest in statistics on simultaneous cores (see~\Cref{sec:motivation}), the purpose of this paper is to give a formula for the average norm of a weight in a highest weight representation.

\begin{theorem}
For $\gf$ a complex simple Lie algebra with $\V_\lambda$ its finite-dimensional irreducible representation of highest weight $\lambda$, the expected norm of a weight in $\V_\lambda$ is
\[\Expt{\mu \in \V_\lambda}{\langle \mu,\mu\rangle}=\frac{1}{\mathrm{dim}(\V_\lambda)} \sum_{\mu \in \V_\lambda} \mathrm{dim}(V_\lambda(\mu)) \langle \mu,\mu\rangle = \frac{1}{h+1} \big\langle \lambda, \lambda + 2\rho\big\rangle,\] where $\mathrm{dim}(V_\lambda(\mu))$ is the multiplicity of $\mu$ in $\V_\lambda$ and $h$ is the Coxeter number of $\gf$.
\label{thm:main_thm}
\end{theorem}

\Cref{thm:main_thm} is illustrated in~\Cref{fig:example1,fig:example2} for highest weight representations for $\mathfrak{sl}_3$ and for $\mathfrak{sp}_4$.

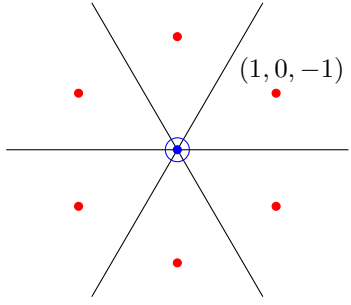
\begin{figure}[htbp]
\begin{center}
\raisebox{-0.6\height}{\begin{tikzpicture}[scale=1.5]
\draw[black,thin] (-1.5,0) -- (1.5,0);
\draw[black,thin] (.75,1.299) -- (-.75,-1.299);
\draw[black,thin] (-.75,1.299) -- (.75,-1.299);
\filldraw[red] (0,1) circle[radius=1pt];
\node at (1,.7) (a) {$(1,0,-1)$};
\filldraw[red] (0,-1) circle[radius=1pt];
\filldraw[red] (.866,.5) circle[radius=1pt];
\filldraw[red] (.866,-.5) circle[radius=1pt];
\filldraw[red] (-.866,.5) circle[radius=1pt];
\filldraw[red] (-.866,-.5) circle[radius=1pt];
\filldraw[blue] (0,0) circle[radius=1pt];
\draw[blue] (0,0) circle[radius=3pt];
\end{tikzpicture}}
\end{center}
\caption{In $\mathfrak{sl}_3$, the \textcolor{black}{eight} weights in $V_\lambda$ with $\lambda=(1,0,-1)$.  The average norm is $\frac{\textcolor{red}{6} \cdot 2 + \textcolor{blue}{2} \cdot 0}{\textcolor{black}{8}} = \frac{3}{2} = \frac{1}{3+1} \langle \lambda,\lambda+2\rho \rangle$.}
\label{fig:example1}
\end{figure}

\begin{figure}[htbp]
\begin{center}
\raisebox{-0.6\height}{\begin{tikzpicture}[scale=.8]
\draw[black,thin] (-2.5,0) -- (2.5,0);
\draw[black,thin] (0,-2.5) -- (0,2.5);
\draw[black,thin] (-1.77,-1.77) -- (1.77,1.77);
\draw[black,thin] (-1.77,1.77) -- (1.77,-1.77);
\filldraw[red] (2,1) circle[radius=1pt];
\node at (2.1,1.3) (a) {$(2,1)$};
\node at (1.2,.4) (a) {$(1,0)$};
\filldraw[red] (-2,1) circle[radius=1pt];
\filldraw[red] (-2,-1) circle[radius=1pt];
\filldraw[red] (2,-1) circle[radius=1pt];
\filldraw[red] (1,2) circle[radius=1pt];
\filldraw[red] (1,-2) circle[radius=1pt];
\filldraw[red] (-1,2) circle[radius=1pt];
\filldraw[red] (-1,-2) circle[radius=1pt];
\filldraw[blue] (1,0) circle[radius=1pt];
\filldraw[blue] (0,1) circle[radius=1pt];
\filldraw[blue] (-1,0) circle[radius=1pt];
\filldraw[blue] (0,-1) circle[radius=1pt];
\draw[blue] (0,1) circle[radius=3pt];
\draw[blue] (0,-1) circle[radius=3pt];
\draw[blue] (1,0) circle[radius=3pt];
\draw[blue] (-1,0) circle[radius=3pt];
\end{tikzpicture}}
\end{center}
\caption{In $\mathfrak{sp}_4$, the \textcolor{black}{sixteen} weights in $V_\lambda$ with $\lambda=(2,1)$.  The average norm is $\frac{\textcolor{red}{8} \cdot 5 + \textcolor{blue}{4 \cdot 2} \cdot 1}{\textcolor{black}{16}} = 3 = \frac{1}{4+1} \langle \lambda,\lambda+2\rho \rangle$.}
\label{fig:example2}
\end{figure}
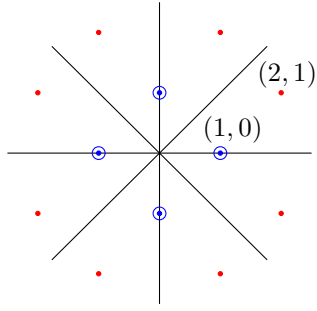

\section{Lie Algebras and their Representation Theory}
\label{sec:lie_algebras}

Recall that the complex simple Lie algebras are classified by their Dynkin diagrams, illustrated in~\Cref{fig:dynkin}.  
Fix a complex simple Lie algebra $\gf$ with Cartan subalgebra $\hf$; all Cartan subalgebras of $\gf$ are conjugate.  Given a complex representation $V:\gf\to \mathfrak{gl}(V)$, we say that the \defn{weight space} for $\mu \in \hf^*$ is the subspace \[V(\mu) = \{ v \in V : H \cdot v = \mu(H) v \text{ for all } H \in \hf\}.\]  The adjoint representation of $\gf$ has non-zero weights called \defn{roots}, and we obtain the \defn{Cartan decomposition} \begin{equation}\gf = \hf \oplus \bigoplus_{\alpha \in \Phi^+} \gf_\alpha \oplus \bigoplus_{\alpha \in \Phi^-} \gf_\alpha,\label{eq:decomp}\end{equation} where $\Phi^+$ and $\Phi^-$ are the positive and negative roots, respectively.  A \defn{simple root} is a positive root that cannot be written as the sum of two positive roots, and we write $\widetilde{\alpha}$ for the highest root.

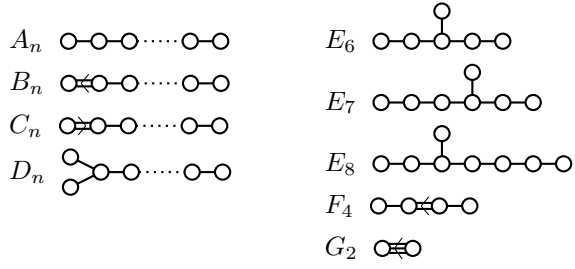
\begin{figure}[tbp]
\begin{center}
\begin{tabular}{cc}
 \parbox[t]{.3\textwidth}{
  \begin{tikzpicture}[scale=.2]
    \draw (-1,0) node[anchor=east]  {$A_n$};
    \draw[thick] (0 cm,0) -- (4 cm,0);
    \draw[dotted,thick] (4 cm,0) -- (8 cm,0);
    \draw[thick] (8 cm,0) -- (10 cm,0);
    \draw[thick,solid,fill=white] (0cm,0) circle (.5cm);
    \draw[thick,solid,fill=white] (2cm,0) circle (.5cm);
    \draw[thick,solid,fill=white] (4cm,0) circle (.5cm);
    \draw[thick,solid,fill=white] (8cm,0) circle (.5cm);
    \draw[thick,solid,fill=white] (10cm,0) circle (.5cm);
  \end{tikzpicture}
  
  \begin{tikzpicture}[scale=.2]
    \draw (-1,0) node[anchor=east]  {$B_n$};
    \draw[thick] (0 cm,-.2) -- (2 cm,-.2);
    \draw[thick] (0 cm,.2) -- (2 cm,.2);
    \draw[thin] (1.3 cm,.6) -- (.8 cm,0);
    \draw[thin] (1.3 cm,-.6) -- (.8 cm,0);
    \draw[thick] (2 cm,0) -- (4 cm,0);
    \draw[dotted,thick] (4 cm,0) -- (8 cm,0);
    \draw[thick] (8 cm,0) -- (10 cm,0);
    \draw[thick,solid,fill=white] (0cm,0) circle (.5cm);
    \draw[thick,solid,fill=white] (2cm,0) circle (.5cm);
    \draw[thick,solid,fill=white] (4cm,0) circle (.5cm);
    \draw[thick,solid,fill=white] (8cm,0) circle (.5cm);
    \draw[thick,solid,fill=white] (10cm,0) circle (.5cm);
  \end{tikzpicture}
  
  \begin{tikzpicture}[scale=.2]
    \draw (-1,0) node[anchor=east]  {$C_n$};
    \draw[thick] (0 cm,-.2) -- (2 cm,-.2);
    \draw[thick] (0 cm,.2) -- (2 cm,.2);
    \draw[thin] (.7 cm,.6) -- (1.2 cm,0);
    \draw[thin] (.7 cm,-.6) -- (1.2 cm,0);
    \draw[thick] (2 cm,0) -- (4 cm,0);
    \draw[dotted,thick] (4 cm,0) -- (8 cm,0);
    \draw[thick] (8 cm,0) -- (10 cm,0);
    \draw[thick,solid,fill=white] (0cm,0) circle (.5cm);
    \draw[thick,solid,fill=white] (2cm,0) circle (.5cm);
    \draw[thick,solid,fill=white] (4cm,0) circle (.5cm);
    \draw[thick,solid,fill=white] (8cm,0) circle (.5cm);
    \draw[thick,solid,fill=white] (10cm,0) circle (.5cm);
  \end{tikzpicture}

  \begin{tikzpicture}[scale=.2]
    \draw (-1,0) node[anchor=east]  {$D_n$};
    \draw[thick] (0 cm,-1) -- (2 cm,0);
    \draw[thick] (0 cm,1) -- (2 cm,0);
    \draw[thick] (2 cm,0) -- (4 cm,0);
    \draw[dotted,thick] (4 cm,0) -- (8 cm,0);
    \draw[thick] (8 cm,0) -- (10 cm,0);
    \draw[thick,solid,fill=white] (0cm,-1) circle (.5cm);
    \draw[thick,solid,fill=white] (0cm,1) circle (.5cm);
    \draw[thick,solid,fill=white] (2cm,0) circle (.5cm);
    \draw[thick,solid,fill=white] (4cm,0) circle (.5cm);
    \draw[thick,solid,fill=white] (8cm,0) circle (.5cm);
    \draw[thick,solid,fill=white] (10cm,0) circle (.5cm);
  \end{tikzpicture}}
  & 
  \parbox[t]{.3\textwidth}{   \begin{tikzpicture}[scale=.2]
    \draw (-1,0) node[anchor=east]  {$E_6$};
    \draw[thick] (0 cm,0) -- (8 cm,0);
    \draw[thick] (4 cm,0) -- (4 cm,2);
    \draw[thick,solid,fill=white] (0cm,0) circle (.5cm);
    \draw[thick,solid,fill=white] (2cm,0) circle (.5cm);
    \draw[thick,solid,fill=white] (4cm,0) circle (.5cm);
    \draw[thick,solid,fill=white] (6cm,0) circle (.5cm);
    \draw[thick,solid,fill=white] (8cm,0) circle (.5cm);
    \draw[thick,solid,fill=white] (4cm,2) circle (.5cm);
  \end{tikzpicture}
  
    \begin{tikzpicture}[scale=.2]
    \draw (-1,0) node[anchor=east]  {$E_7$};
    \draw[thick] (0 cm,0) -- (10 cm,0);
    \draw[thick] (6 cm,0) -- (6 cm,2);
    \draw[thick,solid,fill=white] (0cm,0) circle (.5cm);
    \draw[thick,solid,fill=white] (2cm,0) circle (.5cm);
    \draw[thick,solid,fill=white] (4cm,0) circle (.5cm);
    \draw[thick,solid,fill=white] (6cm,0) circle (.5cm);
    \draw[thick,solid,fill=white] (8cm,0) circle (.5cm);
    \draw[thick,solid,fill=white] (10cm,0) circle (.5cm);
    \draw[thick,solid,fill=white] (6cm,2) circle (.5cm);
  \end{tikzpicture}
  
      \begin{tikzpicture}[scale=.2]
    \draw (-1,0) node[anchor=east]  {$E_8$};
    \draw[thick] (0 cm,0) -- (12 cm,0);
    \draw[thick] (4 cm,0) -- (4 cm,2);
    \draw[thick,solid,fill=white] (0cm,0) circle (.5cm);
    \draw[thick,solid,fill=white] (2cm,0) circle (.5cm);
    \draw[thick,solid,fill=white] (4cm,0) circle (.5cm);
    \draw[thick,solid,fill=white] (6cm,0) circle (.5cm);
    \draw[thick,solid,fill=white] (8cm,0) circle (.5cm);
    \draw[thick,solid,fill=white] (10cm,0) circle (.5cm);
    \draw[thick,solid,fill=white] (12cm,0) circle (.5cm);
    
    \draw[thick,solid,fill=white] (4cm,2) circle (.5cm);
  \end{tikzpicture}

     \begin{tikzpicture}[scale=.2]
    \draw (-1,0) node[anchor=east]  {$F_4$};
    \draw[thick] (0 cm,0) -- (2 cm,0);
    \draw[thick] (2 cm,-.2) -- (4 cm,-.2);
    
    \draw[thin] (3.3 cm,.6) -- (2.8 cm,0);
    \draw[thin] (3.3 cm,-.6) -- (2.8 cm,0);
    \draw[thick] (2 cm,.2) -- (4 cm,.2);

    \draw[thick] (4 cm,0) -- (6 cm,0);
    \draw[thick,solid,fill=white] (0cm,0) circle (.5cm);
    \draw[thick,solid,fill=white] (2cm,0) circle (.5cm);
    \draw[thick,solid,fill=white] (4cm,0) circle (.5cm);
    \draw[thick,solid,fill=white] (6cm,0) circle (.5cm);
  \end{tikzpicture}
  
    \begin{tikzpicture}[scale=.2]
    \draw (-1,0) node[anchor=east]  {$G_2$};
    \draw[thick] (0 cm,-.3) -- (2 cm,-.3);
    \draw[thick] (0 cm,.3) -- (2 cm,.3);
    \draw[thick] (0 cm,0) -- (2 cm,0);
    \draw[thin] (1.3 cm,.7) -- (.8 cm,0);
    \draw[thin] (1.3 cm,-.7) -- (.8 cm,0);
    \draw[thick,solid,fill=white] (0cm,0) circle (.5cm);
    \draw[thick,solid,fill=white] (2cm,0) circle (.5cm);
  \end{tikzpicture}}
  \end{tabular}
  
\end{center}
\caption{The Dynkin diagrams.}
\label{fig:dynkin}
\end{figure}

The \defn{Killing form} is the nondegenerate symmetric bilinear form defined by \[B(X,Y)=\mathrm{tr}(\mathrm{ad}(X),\mathrm{ad}(Y)).\] We normalize the Killing form so that the norm of a long root is 2, and we will write this normalized form as $\langle \cdot,\cdot \rangle$.  We write $\|\alpha\|^2:=\langle \alpha,\alpha \rangle$.

Restricting the Killing form to $\hf$ and writing $\mathrm{dim}(\hf)=n$ allows us to view weights and roots as points in $\mathbb{R}^n$, and the \defn{Weyl group} of $\mathfrak{g}$ is the reflection group $W$ generated by the reflections perpendicular to the roots $\alpha \in \Phi$.  The \defn{coroot} of a root $\alpha$ is $\chk{\alpha}:=\frac{2\alpha}{\langle \alpha,\alpha \rangle}$. 

 Let $\mathcal{A}[\mathfrak{h}^*]$ be the character ring of formal linear combinations of formal exponentials of weights.  For a weight $\lambda\in\mathfrak{h}^*$, let
  \[A_{\lambda}=\sum_{w\in W}\op{sgn}(w)e^{w(\lambda)} \in \mathcal{A}[\mathfrak{h}^*].\]
  Then $A_{\lambda}$ is alternating with respect to the $W$-action on $\mathcal{A}[\mathfrak{h}^*]$, so that $w(A_{\lambda})=\op{sgn}(w)A_{\lambda},$ and it is also a $W$-alternating function of $\lambda$: $A_{w(\lambda)}=\op{sgn}(w)A_{\lambda}.$

 A weight $\lambda$ is called \defn{regular} if $w(\lambda)\neq \lambda$ for all $e \neq w \in W$---in particular, $A_{\lambda}=0$ if $\lambda$ is not regular.  It is called \defn{integral} if its inner product with every coroot is integral.  The \defn{fundamental weights} are the dual basis to the simple coroots, and a weight is called \defn{dominant} if it is a nonnegative linear combination of the fundamental weights.  Let $\rho$ denote the lowest regular dominant integral weight.

\begin{theorem}[Weyl character formula]
\label{thm:weyl}
 If $\lambda$ is an integral weight, there exists a unique $f_{\lambda}\in\mathcal{A}[\mathfrak{h}^*]$ with
 \[A_{\lambda}=f_{\lambda}A_{\rho}.\]
 Furthermore, if $\lambda$ is dominant, then
 \[f_{\lambda+\rho}=\sum_{\mu\in\mathfrak{h}^*}\mathrm{dim}(V_\lambda(\mu)) e^\mu\]
 is the character of the finite dimensional irreducible representation $\V_\lambda$, where the multiplicity of $\mu$ in $\V_\lambda$ is denoted $\mathrm{dim}(V_\lambda(\mu))$.
\end{theorem}

\subsection{Casimir Elements and the Universal Enveloping Algebra}
\label{sec:casimir}
The Harish-Chandra isomorphism is an isomorphism between the center of the universal enveloping algebra of $\gf$, $Z(U(\gf))$, and $W$-invariant polynomials $S(\hf)^W$.  By the Shephard-Todd-Chevelley theorem, $S(\hf)^W$ is a polynomial algebra with $n$ generators, and the degrees $d_1,d_2,\ldots,d_n$ of these generators play an important numerological role: for example, the highest degree is the \defn{Coxeter number} $h$ (the order of a Coxeter element of $W$), the dimension of the Lie algebra is $\mathrm{dim}(\mathfrak{g})=n(h+1)$, the number of reflections in $W$ is $\sum_{i=1}^n (d_i-1)$, and the number of elements in $W$ is $|W|=\prod_{i=1}^n d_i$.  See also~\Cref{sec:coxeter_cumulants} for further numerology.

We call an element of $Z(U(\gf))$ a \defn{Casimir element}---the Harish-Chandra isomorphism combined with the Shephard-Todd-Chevelley theorem shows that there are $n$ algebraically independent Casimir elements.   Special emphasis is given to the Casimir element of degree two, which may be defined as follows: fixing any basis $\{X_i\}_{i=1}^{\mathrm{dim}(\mathfrak{g})}$, we obtain a dual basis $\{X^i\}_{i=1}^{\mathrm{dim}(\mathfrak{g})}$ using the Killing form, and then define \begin{equation}\Omega = \sum_{i=1}^{\mathrm{dim}(\mathfrak{g})} X_iX^i \in Z(U(\gf)).\label{eq:casimir}\end{equation}

As representations of $\mathfrak{g}$ coincide with modules for its universal enveloping algebra, Schur's lemma implies that since $\Omega$ is in the center of $U(\gf)$, it acts as a scalar on any highest weight representation of $\mathfrak{g}$.  The following well-known theorem explicitly identifies this scalar.
\begin{theorem}
Let $\lambda$ be a dominant weight.  Then $\Omega$ acts as multiplication by $\langle\lambda,\lambda+2\rho\rangle$ on $\V_\lambda$.
\label{thm:casimir_action}
\end{theorem}

\begin{proof}
Using the Cartan decomposition of~\Cref{eq:decomp}, write $\Omega = \sum_{1 \leq i \leq n} H_i H^*_i + \sum_{\alpha \in \Phi} E_\alpha E_{-\alpha}$, where we have chosen $\langle E_\alpha,E_{-\alpha}\rangle = 1$ and $[E_\alpha,E_{-\alpha}]=H_\alpha.$  In the representation $V_\lambda$, we compute \begin{align*}\Omega &= \sum_{1 \leq i \leq n} H_i H^*_i + \sum_{\alpha \in \Phi} E_\alpha E_{-\alpha} = \sum_{1 \leq i \leq n} H_i H^*_i + 2\sum_{\alpha \in \Phi^+} E_{-\alpha} E_{\alpha} +  \sum_{\alpha \in \Phi^+} [E_\alpha, E_{-\alpha}] \\ &=  \sum_{1 \leq i \leq n} H_i H^*_i + 2\sum_{\alpha \in \Phi^+} E_{-\alpha} E_{\alpha} +  \sum_{\alpha \in \Phi^+} H_\alpha. \end{align*}  Acting on a highest weight vector in $\V_\lambda$, the term $ 2\sum_{\alpha \in \Phi^+} E_{-\alpha} E_{\alpha}$ vanishes, leaving only $\sum_{1 \leq i \leq n} \lambda(H_i) \lambda(H^*_i) +  \lambda\left(\sum_{\alpha \in \Phi^+} H_\alpha\right)$.  The first term is now computed as $\langle \lambda,\lambda \rangle$, while the second term gives $\langle \lambda,2\rho\rangle$ for $\rho$ the half sum of the positive roots.
\end{proof}

\section{Motivation: Cores and Ehrhart Theory}
\label{sec:motivation}

In this section we relate a special case of~\Cref{thm:main_thm} (for the first fundamantal weight in type $A$) to the study of simultaneous core partitions.

\subsection{Simultaneous Cores and Armstrong's Conjecture}
An \defn{$a$-core} is an integer partition with no hook-length of size $a$.  The study of simultaneous $(a,b)$-cores---that is, partitions that are both $a$-cores and $b$-cores---is a topic that has recently seen quite a lot of interest from the combinatorics community~\cite{stanley2013catalan,aggarwal2014armstrong}.  When $\gcd(a,b)=1$, Anderson proved that the number of $(a,b)$-cores has the simple expression \[\left|\core(a,b)\right| = \frac{1}{a+b}\binom{a+b}{b}\] by giving a bijection to Dyck paths in an $a \times b$ rectangle~\cite{anderson2002partitions}.  It is well-known that the dominant alcoves in the affine symmetric group $\widetilde{S}_a$ are naturally indexed by $a$-cores, and in this language Anderson's result had previously been proven in the generality of affine Weyl groups by both Haiman and Suter~\cite{haiman1994conjectures,suter1998number}.

\medskip

While investigating the interpretation of $q,t$-statistics and the zeta map using the affine symmetric group~\cite{armstrong2011conjecture,armstrong2014results}, Armstrong was led to conjecture that the expected number of boxes of a simultaneous core (its ``$\size$'') had a beautiful formula.

\begin{theorem}[{Amstrong (conjectured), Johnson (proof)~\cite{johnson2015lattice}}]
\label{eq:drew_exp}
The expected number of boxes of a simultaneous core is given by
\[\Expt{\lambda \in \core(a,b)}{\size(\lambda)} = \frac{(a-1)(b-1)(a+b+1)}{24}.\]
\end{theorem}

\begin{example}
We compute the expected number of boxes for the five simultaneous $(3,4)$-cores
\[ \emptyset,\hspace{1em} {\tiny \yng(1)},\hspace{1em} {\tiny \yng(2)},\hspace{1em} {\tiny \yng(1,1)},\hspace{1em} {\tiny \yng(3,1,1)},\]
as \[\frac{1}{5}\left(0+1+2+2+5\right)=2=\frac{(3-1)(4-1)(3+4+1)}{24}.\]
\end{example}

\Cref{eq:drew_exp} was first proven by Johnson using Ehrhart theory~\cite{johnson2015lattice}.  Building on Johnson's approach in \cite{thiel2017strange}, we showed that the statistic $\size$ could be interpreted as a slight modification of the natural norm on the weight space (see~\Cref{fig:weights_and_reps}), and generalized the result to all simply-laced affine Weyl groups (we now have a generalization to all affine Weyl groups):
 \begin{align}
   \size_b(x):=\frac{h}{2}\left\|x-b\frac{\chk{\rho}}{h}\right\|^2-\frac{h}{2}\left\|\frac{\chk{\rho}}{h}\right\|^2.
   \label{def:size}
  \end{align}

Briefly, by composing the bijection between $a$-cores and dominant alcoves in $\widetilde{S}_a$, and the natural bijection between dominant alcoves and coroots, one obtains a bijection between simultaneous $(a,b)$-cores and $\chk{Q} \cap b \mathcal{A}$---coroot points inside a $b$-fold dilation of the fundamental alcove in $\widetilde{S}_a$.  When $a$ is coprime to $b$, the cyclic symmetry of the affine Dynkin diagram gives rise to an affine isometry that partitions the weights inside $b \mathcal{A}$ into regular orbits, each of which contains a single coroot.  It is therefore enough to consider the \emph{coweights} in the dilation of the fundamental alcove $\chk{\Lambda} \cap b \mathcal{A}$, and one may then apply Ehrhart theory to prove (generalizations of) Armstrong's conjecture.

\subsection{Simultaneous Cores and Highest Weight Representations}
As motivation for~\Cref{thm:main_thm}, we wish to show that the problem of computing the expected number of boxes in a simultaneous $(a,b)$-core is roughly equivalent to computing the expected norm of a weight in a particular highest weight representation.  Having already related cores and $\chk{\Lambda} \cap b \mathcal{A}$, we now wish to find a relation to representations.

In $\mathfrak{sl}_a$, there is a bijection between coweights inside the $b$-fold dilation of the fundamental alcove $b\mathcal{A}$, and coweights in the highest weight representation $\mathfrak{sl}_a(b\omega_1)$, where $\omega_1$ is the first fundamental weight---indeed, both are counted by the binomial coefficient $\binom{a+b}{b}$.  This bijection may be described as follows, and is illustrated in~\Cref{fig:weights_and_reps}.  We first center $b\mathcal{A}$ around the origin by sending \[x \mapsto x- \frac{b\chk{\rho}}{h}.\]  There is a natural bijection between the coweight and coroot lattice defined by \begin{align*} \chk{\Lambda} &\to \chk{Q}\\ x &\mapsto (1-c)x,\end{align*} where $c=(a,a{-}1,\ldots,1) \in \mathfrak{S}_a$ is a long cycle.  This reflects the fact that $(1-c)$ is conjugate to the Cartan matrix, with determinant equal to the index of connection $|\chk{Q}/\chk{\Lambda}|$.

\begin{proposition}
\label{prop:bij}
The composition of these two maps gives the desired bijection \normalfont
\begin{align*}\phi: \chk{\Lambda} \cap b\mathcal{A} &\to \mathfrak{sl}_a(b\omega_1) \\ x &\mapsto (1-c)\left(x - \frac{b\chk{\rho}}{h}\right).
\end{align*}

\end{proposition}

\begin{proof}
The vertices of the polytope $b\mathcal{A}$ are $\{0\}\cup \{b\omega_i\}_{i=1}^{n-1}$, and $\chk{\Lambda} \cap b\mathcal{A}$ is (by definition) all coweight points inside the convex hull of those vertices.  On the other hand, the coweights in $\mathfrak{sl}_a(b\omega_1)$ are exactly those coweights in the convex hull of $\{c^i (b\omega_1)\}_{i=0}^{n-1}$ whose difference from $b\omega_1$ is in the coroot lattice.  It therefore suffices to check that the map $(1-c)\left(x - \frac{b\chk{\rho}}{h}\right)$ takes the vertices $\{0\}\cup \{\omega_i\}_{i=1}^{n-1}$ to the vertices $\{c^i (b\omega_1)\}_{i=0}^{n-1}$.  But this is a simple computation---writing $e_i$ for the usual basis of $\mathbb{R}^n$ so that $\omega_i = \sum_{j=1}^i e_j - \frac{i}{n}\sum_{j=1}^n e_j$ and $\chk{\rho}=\sum_{i=1}^{n-1} \omega_i$, we check \begin{align*}(1{-}c)\left(b\omega_i - \frac{b\chk{\rho}}{h}\right) &= \frac{b(1{-}c)}{n}\left(\sum_{j=1}^i \left(\frac{n{-}1}{2}-i+j\right)e_j-\sum_{j=i+1}^n \left(\frac{n{+}1}{2}+i-j\right)e_j\right) \\ &= \frac{b}{n} \left((n-1)e_i+\sum_{j\neq i} -e_j\right)=c^i(b \omega_1). \end{align*}
\end{proof}


\subsection{Equivalence of~\Cref{eq:drew_exp} and~\Cref{thm:main_thm} for $\mathfrak{sl}_a(b\omega_1)$}
Under the bijection of~\Cref{prop:bij}, we show that computing the expected $\size_b$ on $b\mathcal{A}$ (and hence the expected number of boxes in an $(a,b)$-core) is roughly equivalent to computing the expected norm on $\mathfrak{sl}_a(b\omega_1)$.  Write $x_b = x-b\frac{\rho}{h}$ and compute:

\begin{align*}
\Expt{\mu \in \mathfrak{sl}_a(b\omega_1)}{\langle \mu,\mu\rangle} &= \frac{1}{\binom{a+b}{b}}\sum_{\mu \in \mathfrak{sl}_a(b\omega_1)} \| x\|^2 = \frac{1}{\binom{a+b}{b}}\sum_{\mu \in b\mathcal{A}} \left\|(1-c)x_b\right\|^2\\
&=\frac{1}{\binom{a+b}{b}}\sum_{\mu \in b\mathcal{A}}\left(  2\left\|x_b\right\|^2-2\left\langle cx_b,x_b\right\rangle\right)\\
&=\frac{a^2-1}{6a}+\frac{1}{\binom{a+b}{b}}\left(\frac{4}{a}\sum_{\mu \in b\mathcal{A}}\size_b(x)-2\sum_{\mu \in b\mathcal{A}}\left\langle cx_b,x_b\right\rangle\right)
\end{align*}

It is slightly surprising, but follows from an Ehrhart-theoretic computation of roughly the same degree of difficulty as the computation of the expectation of $\size_b$ (the only difficulty arising from a repeated root), that \[\frac{1}{\binom{a+b}{b}} \sum_{\mu \in b\mathcal{A}}\left\langle cx_b,x_b\right\rangle = \frac{(a-5)(a-1)b(a+b)}{12a(a+1)},\]
so that

\begin{align*}
\Expt{\mu \in \mathfrak{sl}_a(b\omega_1)}{\langle \mu,\mu\rangle} &= \frac{a^2-1}{6a} + \frac{4}{a} \frac{(a-1)(b-1)(a+b+1)}{24} - 2 \frac{(a-5)(a-1)b(a+b)}{12a(a+1)} \\ &= \frac{(a-1)b(a+b)}{a(a+1)}.
\end{align*}

On the other hand---since the representation is multiplicity-free---we could have applied Ehrhart theory directly to compute the expected norm of a weight in $\mathfrak{sl}_a(b\omega_1)$.  We conclude that computing $\Expt{\mu \in \mathfrak{sl}_a(b\omega_1)}{\langle \mu,\mu\rangle}$ is roughly equivalent to computing the expectation of size on simultaneous $(a,b)$-cores.

\medskip
Thus, given the success of studying moments of norms of weights in $b\mathcal{A}$, we found it a reasonable extension to ask for the expected norm of a weight in a highest weight representation.  

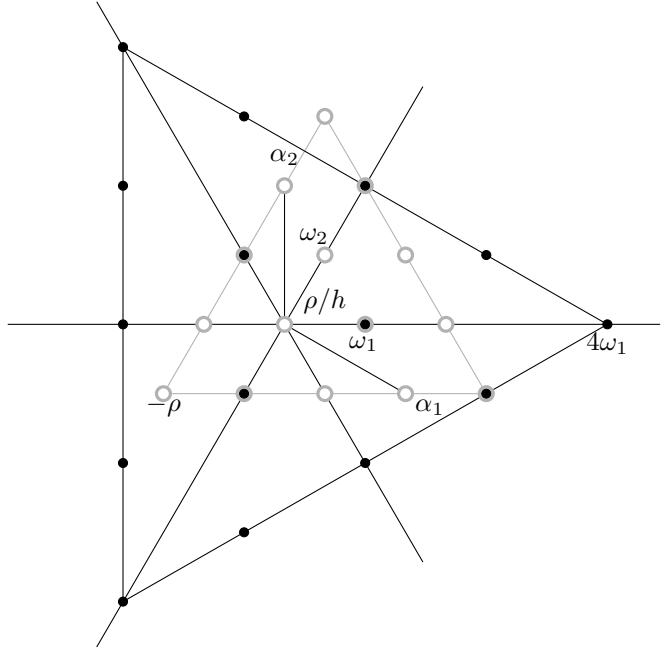
\begin{figure}[htbp]
\begin{center}
\raisebox{-0.6\height}{
\begin{tikzpicture}[scale=1.3]
\draw[->,black,thin] (0,0) -- (0,1.41421);
\draw[->,black,thin] (0,0) -- (1.2247,-.70711);
\draw[black,thin] (-2.8,0) -- (3.8,0);
\draw[black,thin] (1.4,2.4249) -- (-1.9,-3.291);
\draw[black,thin] (-1.9,3.291) -- (1.4,-2.4249);
\tikzmath{\l = -1.225+.8165*4+.40825*0;};
\tikzmath{\k = -.7071+.7071*0;};
\tikzmath{\m = -1.225+.8165*0+.40825*4;};
\tikzmath{\n = -.7071+.7071*4;};
\tikzmath{\o = -1.225+.8165*0+.40825*0;};
\tikzmath{\p = -.7071+.7071*0;};
\draw[gray!60,thin] (\l,\k) -- (\m,\n) -- (\o,\p) -- (\l,\k);
\tikzmath{\l = 3.266-1.225*0;};
\tikzmath{\k = 1.414*0-.7071*0;};
\tikzmath{\m = 3.266-1.225*4;};
\tikzmath{\n = 1.414*4-.7071*4;};
\tikzmath{\o = 3.266-1.225*4;};
\tikzmath{\p = 1.414*0-.7071*4;};
\draw[black,thin] (\l,\k) -- (\m,\n) -- (\o,\p) -- (\l,\k);

   	\foreach \i in {0,...,4}{
      \tikzmath{\ii=4-\i;};
	  \foreach \j in {0,...,\ii}{
		\tikzmath{\l = -1.225+.8165*\i+.40825*\j;};
        \tikzmath{\k = -.7071+.7071*\j;};
        \filldraw[color=gray!60, fill=white, very thick] (\l,\k) circle[radius=2pt];
      };
	};
 	\foreach \i in {0,...,4}{
	  \foreach \j in {\i,...,4}{
		\tikzmath{\l = 3.266-1.225*\j;};
        \tikzmath{\k = 1.414*\i-.7071*\j;};
        \filldraw (\l,\k) circle[radius=1.3pt];
      };
	};
\node at (-1.225,-.83) (a) {$-\rho$};
\node at (.40825,.2357) (b) {$\rho/h$};
\node at (.8,-.2) (c) {$\omega_1$};
\node at (.3,.9) (d) {$\omega_2$};
\node at (3.266,-.2) (a) {$4\omega_1$};
\node[black] at (1.47,-.85) (b) {$\alpha_1$};
\node[black] at (0,1.7) (c) {$\alpha_2$};
\end{tikzpicture}}
\end{center}
\caption{The weights inside a $4$-fold dilation of the fundamental alcove in $\mathfrak{sl}_3$ are drawn as gray circles inside the gray triangle, while the weights in the representation $V_{4\omega_1}$ are drawn as black disks.  The statistic $\size$ on the weights in the dilation of the fundamental alcove is a quadratic form that is a slight modification of the norm.}
\label{fig:weights_and_reps}
\end{figure}

\section{Proof of~\Cref{thm:main_thm} using the Weyl Character Formula}

In this section we use the polynomial method and the Weyl character formula (\Cref{thm:weyl}) to give an elementary, uniform proof of~\Cref{thm:main_thm} in all types.  We first prove polynomiality of sums of polynomial functions in the weights in $\V_\lambda$. 


\begin{theorem}\label{avgpoly}
 Suppose that $\lambda$ is a dominant integral weight of $\mathfrak{h}$ and let $\V_{\lambda}$ be the finite dimensional irreducible representation of $\mathfrak{g}$ of highest weight $\lambda$.
 For a weight $\mu$, let $\mathrm{dim}(\V_\lambda(\mu))$ be the multiplicity of $\mu$ in $\V_\lambda$.
 Let $P\in S(\mathfrak{h})^W$ be a $W$-invariant polynomial on $\mathfrak{h}^*$ of degree $d$. Then
 \[S(P,\lambda):=\frac{1}{\op{dim}(\V_\lambda)}\sum_{\mu\in\mathfrak{h}^*}\mathrm{dim}(\V_\lambda(\mu))P(\mu)\]
 is a polynomial in $\lambda$ of at most degree $d$. It is a $W$-invariant polynomial in variables given by $\lambda+\rho$.
\end{theorem}
\begin{proof}
 We follow the ideas and notation of the derivation of the Weyl dimension formula from the Weyl character formula in \cite[Section 7.1.2]{goodman2009symmetry}.
 
We write $f_{\lambda+\rho}=\frac{A_{\lambda+\rho}}{A_{\rho}}$.
We define the linear differential operator \begin{align*}N:\mathcal{A}[\mathfrak{h}^*]&\rightarrow \mathcal{A}[\mathfrak{h}^*] \\ N(e^{\lambda})&=P(\lambda)e^\lambda,\end{align*}
and the linear evaluation map \begin{align*}\epsilon:\mathcal{A}[\mathfrak{h}^*]&\rightarrow\mathbb{C}\\
\epsilon(e^\lambda)&=1.\end{align*}
Then we have that
\[\frac{\epsilon(N(f_{\lambda+\rho}))}{\epsilon(f_{\lambda+\rho})}=\frac{\sum_{\mu\in\mathfrak{h}^*}\mathrm{dim}(\V_\lambda(\mu))P(\mu)}{\sum_{\mu\in\mathfrak{h}^*}\mathrm{dim}(\V_\lambda(\mu))}=S(P,\lambda).\]
Write $A_\rho=\prod_{\alpha\in\Phi^+}(e^{\alpha/2}-e^{-\alpha/2})$. Now $N$ is a differential operator of degree $d$, so the quotient rule implies that
\begin{align*}
 N(f_\lambda)&=N\left(\frac{A_\lambda}{A_\rho}\right)=N\left(\frac{A_\lambda}{\prod_{\alpha\in\Phi^+}(e^{\alpha/2}-e^{-\alpha/2})}\right)\\
 &=\sum \frac{B_{\lambda,\mathbf{m}}}{\prod_{\alpha\in\Phi^+}(e^{\alpha/2}-e^{-\alpha/2})^{m_\alpha+1}},
\end{align*}
where the sum is over all $\mathbf{m}=(m_\alpha)\in[d]^{\Phi^+}$ such that $\sum_{\alpha\in\Phi^+}m_\alpha\leq d$, and where $B_{\lambda,\mathbf{m}}\in\mathcal{A}[\mathfrak{h}^*]$ has coefficients that are polynomials in $\lambda$ of degree at most $d-\sum_{\alpha\in\Phi^+}m_\alpha$.
Using L'H\^{o}pital's rule we see that $\epsilon(N(f_\lambda))$ is a polynomial in $\lambda$ of degree at most $|\Phi^+|+d$. It is alternating in $\lambda$.
In particular, $\epsilon(N(f_\lambda))=0$ if $\lambda$ is not regular. So for every $\alpha\in\Phi^+$ the linear factor $\langle\lambda,\alpha\rangle$ divides the polynomial $\epsilon(N(f_\lambda))$.

Consider the special case where $P=1$, so that $N$ is the identity and $d=0$. Then this implies that $\epsilon(N(f_\lambda))=\epsilon(f_\lambda)=C\prod_{\alpha\in\Phi^+}\langle\lambda,\alpha\rangle$ for a constant $C$.
We have that $f_\rho=e^0$ so that $C=\frac{1}{\prod_{\alpha\in\Phi^+}\langle\rho,\alpha\rangle}$.

For any $W$-invariant polynomial $P\in S(\mathfrak{h})^W$, the polynomial $\epsilon(f_\lambda)$ therefore divides the polynomial $\epsilon(N(f_\lambda))$, so that the quotient $\frac{\epsilon(N(f_\lambda))}{\epsilon(f_\lambda)}$ is a polynomial in $\lambda$ of degree at most $d$.
It is $W$-invariant, since both $\epsilon(N(f_\lambda))$ and $\epsilon(f_\lambda)$ are alternating in $\lambda$.
Thus $S(\lambda,P)=\frac{\epsilon(N(f_{\lambda+\rho}))}{\epsilon(f_{\lambda+\rho})}$ is given by a $W$-invariant polynomial of degree at most $d$ in $\lambda+\rho$.
\end{proof}

\begin{example}
Consider $\mathfrak{g}=\mathfrak{sl}_2$ and $P=\|\cdot\|^2$, where $\|\cdot \|^2$ is given by the $\mathfrak{S}_2$-equivariant polynomial $\|(x,-x)\|^2 = 2x^2$.  Since the dominant weights are given by $\lambda = (m,-m)$ for $m \in \mathbb{N}$, and since the weights in $V_\lambda$ are exactly of the form $(m,-m),(m-2,-m+2),\ldots,(-m,m)$ (with no multiplicity), we compute that \begin{align*}S(\|\cdot \|^2,\lambda) = \frac{2}{m+1} \sum_{i=0}^m (m-2i)^2 = \frac{2}{3} m(m+2) &= \frac{1}{3}\left(2(m+1)^2-2\right),\\ &=\frac{1}{h+1}(\|\lambda+\rho\|^2-\|\rho\|^2), \end{align*}  which is $\mathfrak{S}_2$-equivariant as a polynomial in $m+1$.
\end{example}

 \begin{proof}[First proof of~\Cref{thm:main_thm}]  For this proof, we use the normalization of the Killing form that $\|\widetilde{\alpha} \|^2=\frac{1}{g}$, where $g$ is the dual Coxeter number.
  By Theorem \ref{avgpoly}, $S(\|\cdot\|,\lambda)$ is a $W$-invariant polynomial of degree at most $2$ in $\lambda+\rho$ so that $S(\|\cdot\|^2,\lambda)=a+b\|\lambda+\rho\|^2$ for some $a,b\in\CC$.
  We have that $S(\|\cdot\|^2,0)=0$, so $a=-b\|\rho\|^2$.
  Furthermore, if $\widetilde{\alpha}\in\Phi^+$ is the highest root, then $V_{\widetilde{\alpha}}$ is the adjoint representation of $\mathfrak{g}$, so we get
  \[S(\|\cdot\|^2,\widetilde{\alpha})=\frac{1}{n(h+1)}\sum_{\alpha\in\Phi}\|\alpha\|^2=\frac{1}{h+1}\]
  using $\sum_{\alpha\in\Phi}\|\alpha\|^2=n$ for the Killing form \cite{brown1964remark}.
  So $\frac{1}{h+1}=b(\|\widetilde{\alpha}+\rho\|^2-\|\rho\|^2)=b$, using that $\|\widetilde{\alpha}+\rho\|^2-\|\rho\|^2$ is the Casimir eigenvalue on the adjoint representation and therefore equals $1$.
  We conclude that \[S(\|\cdot\|^2,\lambda)=\frac{1}{h+1}(\|\lambda+\rho\|^2-\|\rho\|^2).\qedhere\]
 \end{proof}


\section{A Combinatorial Proof in Type $A$}
\label{sec:typea}

In this section, we make the polynomiality argument of the previous section more concrete using the combinatorics of the representation theory of $\mathfrak{sl}_n$.  Fix $\mathfrak{g}=\mathfrak{sl}_{n}$ and $\omega_i=\sum_{j=1}^i e_i$.  In $\mathfrak{sl}_n$, dominant weights of $\mathfrak{h}$ may be parametrized as integer partitions \[\lambda=[\lambda_1 \geq \lambda_2 \geq \cdots \geq \lambda_n] \vdash m,\] where parts $\lambda_i$ may be equal to zero.  Fix a highest weight $\lambda$, and write
\begin{align*}
	\overline{\lambda_i} = \lambda_i - \frac{|\lambda|}{n}, \text{ } \overline{\lambda}=\left[\overline{\lambda_1}\geq \cdots \geq \overline{\lambda_n}\right] \text{, and }
	\rho = \left[n-1,n-2,\ldots,1,0\right].
\end{align*}

With these conventions, weights $\mu$ in the highest weight representation $\mathfrak{sl}_{n}(\lambda)$ may be thought of as certain points in $\mathbb{R}^n$ with positive entries and sum equal to $m$.  Combinatorially, the multiplicity of $\mu$ in $\mathfrak{sl}_{n}(\lambda)$ is well-known to be given by the number of semistandard tableaux of shape $\lambda$ on the alphabet $[n]$ with content $\mu$; $m$ is just the number of boxes in the Ferrers shape $\lambda$:
\[\mathrm{ch}\left(\mathfrak{sl}_{n,\lambda}\right) = s_\lambda(x_1,\ldots,x_n)= \sum_{\substack{T \text{ semistandard} \\ \text{ of shape }\lambda}} \mathbf{x}^T,\]
where $\mathbf{x}^T = \prod_{i=1}^n x_i^{\left|\{i \in T\}\right|}$ and $s_\lambda(x_1,\ldots,x_n)$ is a Schur polynomial.  As a simple consequence of this combinatorial description of the character, we have Weyl's ``interlacing'' multiplicity-free formula for the branching of the representation $\mathfrak{sl}_{n,\lambda}$ to $\mathfrak{sl}_{n-1}$:\[\mathfrak{sl}_{n,\lambda} = \bigoplus_{\mu} \mathfrak{sl}_{n-1,\mu}, \text{ where } \lambda_1 \geq \mu_1 \geq \lambda_2 \geq \cdots \geq \mu_{n-1} \geq \lambda_n.\]  By symmetry of the Schur function, and since this branching rule exactly peels off the boxes containing the entry $n$ (corresponding to the value of the coordinate $x_n$), one could imagine using this formula to determine the norm by computing
\[\sum_{\mu \in V_\lambda} \mathrm{dim}\left(\mathfrak{sl}_{n,\lambda}(\mu)\right) \langle \mu,\mu\rangle = n\sum_{\mu} \mathrm{dim}\left(\mathfrak{sl}_{n-1,\mu}\right)\left(|\lambda|-|\mu|\right)^2.\]
We do not follow this approach here, but instead isolate the boxes containing the entry $n$ using the Pieri rule and an inclusion-exclusion argument.


\begin{theorem}
Let $\mathfrak{g}=\mathfrak{sl}_{n}$.  Suppose that $\lambda$ is a dominant weight of $\mathfrak{h}$ and let $\mathfrak{sl}_{n,\lambda}$ be the finite dimensional irreducible representation of $\mathfrak{sl}_n$ of highest weight $\lambda$. Then
\[\frac{1}{\mathrm{dim}(\mathfrak{sl}_{n,\lambda})} \sum_{\mu \in \mathfrak{sl}_{n}(\lambda)} \mathrm{dim}(\mathfrak{sl}_{n,\lambda}(\mu)) \langle \overline{\mu},\overline{\mu}\rangle = \frac{1}{h+1} \big\langle \overline{\lambda}, \overline{\lambda + 2\rho}\big\rangle.\]
\end{theorem}

\begin{proof}[Second proof of~\Cref{thm:main_thm}, valid in type $A$]
	Some care is needed when we compute the length of $\mu \in \mathfrak{sl}_{n}(\lambda)$---we wish to compute the length of the \emph{normalized} weight $\overline{\mu}$.
Of course, there is a simple relationship between the length of $\mu$ and of $\overline{\mu}$: $\langle \overline{\mu},\overline{\mu} \rangle = \langle \mu,\mu\rangle- \frac{m^2}{n},$ where $m=\langle \mu, [1]^n \rangle$ (constant for all $\mu \in \mathfrak{sl}_{n}(\lambda)$).  We may therefore compute with unnormalized weights using the relationship
\begin{align*}\frac{1}{\mathrm{dim}(\mathfrak{sl}_{n,\lambda})} \sum_{\mu \in \mathfrak{sl}_{n,\lambda}} \mathrm{dim}(\mathfrak{sl}_{n,\lambda}(\mu)) \langle \overline{\mu},\overline{\mu}\rangle =& -\frac{m^2}{n}\\&+\frac{1}{\mathrm{dim}(\mathfrak{sl}_{n,\lambda})} \sum_{\mu \in \mathfrak{sl}_{n,\lambda}} \mathrm{dim}(\mathfrak{sl}_{n,\lambda}(\mu)) \langle \mu,\mu\rangle.\end{align*} 

Define $n$ new partitions
	\[\lambda^{(i)} = \left[\lambda_1+1 \geq \lambda_2+1 \geq \cdots \geq \lambda_{i-1}+1 \geq \lambda_{i+1} \geq \cdots \geq \lambda_n\right] \text{ for } 1 \leq i \leq n.\]
Using the fact that Schur polynomials in the variables $x_i$ are symmetric, conditioning on which boxes of $\lambda$ contain the entry $n$, and using the Pieri rule allows us to write

\begin{multline*}\frac{1}{\mathrm{dim}(\mathfrak{sl}_{n,\lambda})}\sum_{\mu \in \mathfrak{sl}_{n,\lambda}} \mathrm{dim}(\mathfrak{sl}_{n,\lambda}(\mu)) \langle \overline{\mu},\overline{\mu}\rangle=
	-\frac{m^2}{n}+\\+\overbrace{\frac{n}{\mathrm{dim}(\mathfrak{sl}_{n,\lambda})}}^{\substack{\text{Schur polynomial}\\ \text{symmetry}}} \sum_{j=1}^n \overbrace{(-1)^{j+1}}^{\substack{\text{inclusion-}\\\text{exclusion}}} \Big( \underbrace{s_{\lambda^{(j)}}([1]^{n-1}) \sum_{i=0}^{\lambda_j-(j-1)} h_i([1]^{n-1})}_{\text{Pieri rule; leftover boxes contain $n$}} \underbrace{\left(\lambda_j-(j-1)-i\right)^2}_{\substack{\text{contribution to norm}\\ \text{ of boxes containing $n$}}} \Big),
\end{multline*}
where the alternating sum reflects an inclusion-exclusion argument that removes the over-count of those partitions that aren't contained in $\lambda$.  The point is that the numerator has now been expressed as a polynomial.

\begin{example}
As in \Cref{fig:example1}, let $\lambda = (2,1)$ and $n=3$.   We consider all eight semistandard tableaux of shape $(2,1)$ with entries at most 3: \[ \begin{tabular}{c}\young(13,3)\\[\smallskipamount] \young(23,3)\end{tabular} \hspace{3em} \begin{tabular}{c} \young(11,3) \\[\smallskipamount] \young(12,3) \\[\smallskipamount] \young(22,3) \\[\medskipamount] \young(13,2)  \end{tabular} \hspace{3em} \begin{tabular}{c}\young(11,2) \\[\smallskipamount] \young(12,2) \end{tabular}.\]  
By symmetry of the Schur function, \[\sum_T \sum_{i=1}^3 \left(\text{number of $i$s in T}\right)^2 = 3\sum_T \left(\text{number of $3$s in T}\right)^2,\]
and so we should group the tableaux by the number of boxes containing the entry $n=3$.  The sum above uses the Pieri rule to do this, expressing these eight tableaux as

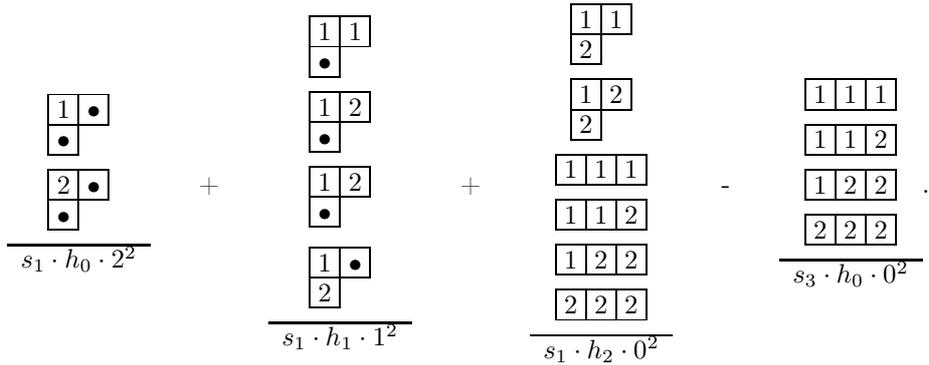
\begin{center}
\begin{tabular}{c}\young(1\bullet,\bullet)\\[\smallskipamount] \young(2\bullet,\bullet) \\[\smallskipamount] \hline $s_1 \cdot h_0 \cdot 2^2$ \end{tabular} \hspace{1.5em}+\hspace{1.5em} \begin{tabular}{c} \young(11,\bullet) \\[\smallskipamount] \young(12,\bullet) \\[\smallskipamount] \young(12,\bullet) \\[\medskipamount] \young(1\bullet,2) \\[\smallskipamount]\hline $s_1 \cdot h_1 \cdot 1^2$  \end{tabular} \hspace{1.5em}+\hspace{1.5em} \begin{tabular}{c}\young(11,2) \\[\smallskipamount] \young(12,2) \\[\smallskipamount] \young(111) \\[\smallskipamount] \young(112) \\[\smallskipamount] \young(122) \\[\smallskipamount] \young(222) \\[\smallskipamount] \hline $s_1 \cdot h_2 \cdot 0^2$ \end{tabular} \hspace{1.5em}-\hspace{1.5em} \begin{tabular}{c} \young(111) \\[\smallskipamount] \young(112) \\[\smallskipamount] \young(122) \\[\smallskipamount] \young(222) \\[\smallskipamount] \hline $s_3 \cdot h_0 \cdot 0^2$ \end{tabular}.
\end{center}
\end{example}

We have the evaluations of the Schur and homogeneous functions at $x_i=1$
\begin{align}
	s_\lambda ([1]^n) = \prod_{1 \leq i < j \leq n} \frac{\lambda_i-\lambda_j+j-i}{j-i} \text{ and }
	h_i ([1]^n) = \binom{n+i-1}{i}.
\end{align}


Dividing by $\mathrm{dim}(\mathfrak{sl}_{n}(\lambda))=s_{\lambda}([1]^n)$, using the formulas above, performing the obvious cancellations, and explicitly evaluating the sums, we obtain 


\begin{multline}\frac{1}{\mathrm{dim}(\mathfrak{sl}_{n,\lambda})}\sum_{\mu \in \mathfrak{sl}_{n,\lambda}} \mathrm{dim}(\mathfrak{sl}_{n,\lambda}(\mu)) \langle \overline{\mu},\overline{\mu}\rangle=\\
	=-\frac{m^2}{n}+n! \sum_{j=1}^n\left(\prod_{\substack{1 \leq i \leq n \\ i\neq j}} \frac{1}{\lambda_j-\lambda_i+i-j}\right) \sum_{i=0}^{\lambda_j-(j-1)} \binom{n+i-2}{i} \left(\lambda_j-(j-1)-i\right)^2\\
	=-\frac{m^2}{n}+\frac{n!}{n+1}\sum_{j=1}^n  \left(\prod_{\substack{1 \leq i \leq n \\ i\neq j}} \frac{1}{\lambda_j-\lambda_i+i-j}\right) \binom{n+\lambda_j-j}{\lambda_j-j}\left(n+2(\lambda_j-j)+1\right).\label{eq:rhs}
\end{multline}



On the other hand,
\begin{align}
\frac{1}{h+1}\langle \overline{\lambda}, \overline{\lambda+2\rho}\rangle &= \frac{1}{n+1}\left\langle \lambda - \frac{m}{n}[1]^n, \lambda- \frac{m}{n}[1]^n+2\rho - (n-1)[1]^n \right \rangle  \\
&=-\frac{m^2-mn+n^2}{n(n+1)}+\frac{1}{n+1}\left(\sum_{i=1}^n\lambda_i^2 +2(n-i)\lambda_i\right)\label{eq:lhs}
\end{align}

Setting~\Cref{eq:lhs,eq:rhs} equal, multiplying by $(n+1)$, pushing the constants to one side, and writing

\begin{align*}
	x_j = \lambda_j - j \text{  and  }
	P(x_j)= \left(n+2x_j+1\right) \prod_{i=1}^n (x_j+i),
\end{align*}
 we must check that

\begin{align}
	2\left(\sum_{1\leq i \leq j \leq n} x_ix_j\right) +  (1 + n)^2 \left(\sum_{i=1}^n x_i\right)+\frac{n (n + 1) (3n^2+5n+4)}{12}\label{eq:lhs2}\\
	= \sum_{j=1}^n  \left(\prod_{\substack{1 \leq i \leq n \\ i\neq j}} \frac{1}{x_j-x_i}\right) P(x_j).\label{eq:rhs2}
\end{align}

Treating the $x_i$ as formal variables, \Cref{eq:lhs2} is clearly a symmetric polynomial of degree $2$ in the $x_j$.  For any $1 \leq i < j \leq n$, the $i$th and $j$th terms of the sum in \Cref{eq:rhs2} are the only two terms with $x_i-x_j$ in the denominator---and the sum of these terms is multiplied by $P(x_j)-P(x_i)$.  This is true for any $i$ and $j$, and so all residues cancel.  We conclude that since $P(x_j)$ has degree $n+1$, \Cref{eq:rhs2} is also a symmetric polynomial in the $x_j$ of degree at most $2$.  It remains to confirm that these are the same degree $2$ symmetric function.


Setting $x_{i}=-i$ for $1 \leq i \leq n-2$, every term except the last two vanish in \Cref{eq:rhs2}.  These remaining terms simplify to
\begin{align*}
&\left(\frac{n+2x_n+1}{x_n-x_{n-1}} \frac{\prod_{i=1}^n x_n+i}{\prod_{i=1}^{n-2}x_n+i}\right)
+\left(
\frac{n+2x_{n-1}+1}{x_{n-1}-x_{n}} \frac{\prod_{i=1}^n x_{n-1}+i}{\prod_{i=1}^{n-2}x_{n-1}+i}\right)\\
&=
\frac{\left(n{+}2x_n{+}1\right)(x_n{+}n{-}1)(x_n{+}n)-\left(n{+}2x_{n-1}{+}1\right)(x_{n-1}{+}n-1)(x_{n-1}{+}n)}{x_n{-}x_{n-1}}\\
&=2\left(x_{n-1}^2+x_{n-2}^2\right)+2x_nx_{n-1}+(5n-1)(x_n+x_{n-1})+4n^2-n-1,
\end{align*}

which proves that the coefficients of $x_i^2$ and $x_ix_j$ in \Cref{eq:rhs2} agree with those in \Cref{eq:lhs2}:
\begin{align}	&\sum_{j=1}^n  \left(\prod_{\substack{1 \leq i \leq n \\ i\neq j}} \frac{1}{x_j-x_i}\right) P(x_j) =2\left(\sum_{1\leq i \leq j \leq n} x_ix_j\right) +  C_1 \left(\sum_{i=1}^n x_i\right)+C_0\label{eq:rhs3}
.\end{align}

We now determine $C_1$.  Setting $x_i=-i$ for $1 \leq i \leq n$ makes every term in \Cref{eq:rhs2} vanish; setting $x_i=-i+1$ leaves only the first term, which simplifies to $n(n+1)$.  Specializing $x_i$ to these values in \Cref{eq:rhs3} and subtracting, we obtain
\begin{align*}
	&\left(\frac{n(n+1)(n+2)(3n+1)}{12}-\frac{n(n+1)}{2}C_1+C_0\right)\\&-\left(\frac{(n-1)n(n+1)(3n-2)}{12}-\frac{(n-1)n}{2}C_1+C_0\right)=n^3+n^2-nC_1.
\end{align*}
Equating $n^3+n^2-nC_1 = 0-n(n+1)$, we obtain $C_1=(n+1)^2$.

Again setting $x_i=-i$ for $1 \leq i \leq n$ so that \Cref{eq:rhs2} is 0, we finally determine that $C_0=\frac{n (n + 1) (3n^2+5n+4)}{12}$ by computing
\[
C_0=2\left(\sum_{1\leq i \leq j \leq n} ij\right)-(n{+}1)^2 \left(\sum_{i=1}^n i\right)=\frac{n(n{+}1)(3n^2{+}5n{+}4)}{12}.
\qedhere
\]

\end{proof}

\section{Proof of \Cref{thm:main_thm} using Orthogonal Decompositions}

In this section, we give a conceptual proof of~\Cref{thm:main_thm} (in types other than $A$ and $C$) using the theory of orthogonal decompositions.   Our strategy is to compute the trace of the degree two Casimir element $\Omega$ in two different ways on the representation $\V_\lambda$.

\subsection{Orthogonal Decompositions of Lie algebras}
\label{sec:orthogonal}
The usual decomposition of $\gf$ using a fixed Cartan subalgebra $\hf$ and the adjoint representation is given in~\Cref{eq:decomp}.  Numerologically, this reflects the identity \[n(h+1)=n+nh=\mathrm{dim}(\hf)+|\Phi|.\]
But since $\mathrm{dim}(\hf)=n$ divides $\mathrm{dim}(\gf)=n(h+1)$, we might ask for a \emph{different} decomposition of $\gf$ using a direct sum of $h+1$ Cartan subalgebras:
\begin{equation}
\gf = \bigoplus_{i=0}^{h} \hf_i, \text{ with } \hf_i \text{ a Cartan subalgebra of }\gf \text{ and } \hf_0=\hf.
\end{equation}
In fact, such a decomposition is always possible.  More difficult is to require that these $h+1$ Cartan subalgebras are pairwise orthogonal with respect to the Killing form; such a decomposition is called an \defn{orthogonal decomposition}.  We refer the reader to~\cite{kostrikin1981orthogonal, kostrikin1994orthogonal} for background and references, pausing only to remark that such decompositions have a number of applications, including Thompson's construction of his sporadic simple group from the Lie algebra of type $E_8$~\cite{thompson1976conjugacy} and the construction of mutually unbiased bases for quantum cryptography~\cite{boykin2007mutually}.

We note that Kostant used a dual approach to the related numerological problem of trying to uniformly explain the duality between degrees and the heights of roots~\cite{kostant2009principal}.  Kostant decomposed $\gf$ into direct sum of $n$ irreducible representations of the principal three dimensional simple subalgebra (a distinguished copy of $\mathfrak{sl}_2$ inside $\gf$), reflecting the identity $n(h+1) = \sum_{i=1}^n (2d_i-1)$.

\begin{theorem}[{\cite{kostrikin1994orthogonal}}]
A complex simple Lie algebra $\gf$ has an orthogonal decomposition, except possibly if
\begin{itemize}
\item $\gf=\mathfrak{sl}_n$ for $n$ not a prime power; or if
\item $\gf=\mathfrak{sp}_{2n}$ for $n \neq 2^m$.
\end{itemize}
\label{thm:orthogonal}
\end{theorem}
Although types $A$ and $C$ are usually the easiest Lie algebras to work with, it is widely believed that these classical Lie algebras \emph{do not} have orthogonal decompositions (outside the cases listed above); this problem is wide open, even for $\mathfrak{sl}_6$.

The problem of finding such decompositions was dubbed the \defn{Winnie-the-Pooh problem} in the Russian paper~\cite{kostrikin1981orthogonal}, due to a play on words found in Zahoder's translation of Milne's famous children's book ``Winnie-the-Pooh'' into Russian.  Zahoder's play on words can be interpreted as the sequence of Cartan types $A_5$---corresponding to the smallest open case $\mathfrak{sl}_6$---then $A_6, A_7,$ and $A_8$.  This play on words apparently has no counterpart in Milne's original text, so when translating~\cite{kostrikin1981orthogonal} into English, Queen also translated Zahoder's verse---while managing to preserve the pun~\cite{kostrikin1994orthogonal}.

\begin{proof}[Third proof of~\Cref{thm:main_thm}, valid in types not $A$ or $C$]

Suppose $\gf$ has an orthogonal decomposition $\gf = \bigoplus_{i=0}^{h} \hf_i.$  For each $0\leq i\leq h$, pick an orthonormal basis $\{X_{i,1},\ldots,X_{i,n}\}$ of $\hf_i$.  Then \[\Big\{X_{i,j} : 0 \leq i \leq h\text{ and } 1 \leq j \leq n \Big\}\] is an orthonormal basis of $\gf$, so we may write the degree two Casimir element $\Omega$ as
\[\Omega = \sum_{\substack{0 \leq i \leq h\\ 1 \leq j \leq n }} X_{i,j}^2.\]
We compute the trace of $\Omega$ on $\V_\lambda$ in two different ways.  On the one hand, $\Omega$ acts as the scalar $\langle \lambda,\lambda+2\rho\rangle$ by~\Cref{thm:casimir_action}, so that\[\mathrm{tr}_{\V_\lambda} (\Omega) = \left\langle \lambda,\lambda+2\rho \right\rangle \mathrm{dim}(\V_\lambda).\]
On the other hand, for $0 \leq i \leq h$ define $\Omega_i = \sum_{j=0}^n X_{i,j}^2.$  By definition, $X_{0,j}$ acts as $\mu(X_{0,j})$ on the $\mu$-weight space of $\V_\lambda$, so
\begin{align*}
\mathrm{tr}_{\V_\lambda}(\Omega_0) = \sum_{\mu \in \V_\lambda} \mathrm{dim}(\V_\lambda(\mu)) \sum_{j=1}^n \mu(X_{0,j})^2 =\sum_{\mu \in \V_\lambda} \mathrm{dim}(\V_\lambda(\mu)) \langle \mu,\mu\rangle,
\end{align*}
since $\{X_{0,1},\ldots,X_{0,n}\}$ is an orthonormal basis of $\hf=\hf_0$.  Since every $\hf_i$ is conjugate to $\hf_0$ under an inner automorphism of $\gf$ that leaves the Killing form invariant, we have that $\Omega_i$ is conjugate to $\Omega_0$ for all $0\leq i \leq h$.  Therefore,
\begin{align*}\mathrm{tr}_{\V_\lambda} (\Omega) &= \mathrm{tr}_{\V_\lambda} \left(\sum_{i=0}^h \Omega_i\right) = \sum_{i=0}^h \mathrm{tr}_{\V_\lambda}\left(\Omega_i\right) \\ &= (h+1) \mathrm{tr}_{\V_\lambda}\left(\Omega_0\right) = (h+1)\sum_{\mu \in \V_\lambda} \mathrm{dim}(\V_\lambda(\mu)) \langle \mu,\mu\rangle.\end{align*}  The result now follows from equating the two expressions for $\mathrm{tr}_{\V_\lambda} (\Omega)$.
\end{proof}

By~\Cref{thm:orthogonal}, this proof of~\Cref{thm:main_thm} applies to all types except possibly if $\gf=\mathfrak{sl}_n$ for $n$ composite; or if $\gf=\mathfrak{sp}_{2n}$ for $n \neq 2^m$.

\section{Coxeter Cumulants}
\label{sec:coxeter_cumulants}

Let $[n]_q:=1+q+\cdots+q^{n-1}$ be the uniform distribution on $\{0,1,\ldots,n-1\}$.
\begin{lemma}
Fix $\{a_i\}_{i=1}^k$ and $\{b_i\}_{i=1}^k$ two sets of positive integers with $\prod_{i=1}^k \frac{[a_i]_q}{[b_i]_q}$ a polynomial in $q$, and let $X$ be a random variable with this distribution.  Then the $r$th cumulant of $X$ is \[\kappa_r(X)=\frac{B_r}{r}\sum_{i=1}^n \left(a_i^r-b_i^r\right),\] where $B_r$ is the $r$th Bernoulli number.  In particular, $\kappa_r(X)=0$ for odd $r>1$.
\label{prop:cumulant}
\end{lemma}
\begin{proof}
We first claim that the $r$th cumulant of $[n]_q$ is \begin{equation}\kappa_r([n]_q)=\frac{B_r}{r}(n^r-1).\label{eq:ncum}\end{equation}  Let $X$ be a random variable whose distribution is $[n]_q$, and let \[u(t)=\mathrm{ln}\Expt{}{e^{tX}} = \sum_{r=0}^\infty \kappa_r(X) \frac{t^r}{r!}\] be its cumulant exponential generating function.  We claim that \begin{equation}u(t)=\mathrm{ln}(n)+\sum_{r=1}^\infty \frac{B_r}{r}(n^r-1)\frac{t^r}{r}, \label{eq:fcumulant}\end{equation} so that the coefficients of $\frac{t^r}{r!}$ are as desired.  To prove~\Cref{eq:fcumulant}, we write \[u(t)=\mathrm{ln}\left(\sum_{i=0}^{n-1} e^{it} \right)=\mathrm{ln}\left( \frac{1-e^{nt}}{1-e^t}\right).\]  Taking a derivative of~\Cref{eq:fcumulant} using the fact that $\sum_{r=0}^\infty B_r \frac{t^r}{r!}=\frac{t}{1-e^{-t}}$, we obtain \[\sum_{r=1}^\infty B_r(n^r-1)\frac{t^r-1}{r!} = \frac{1}{t}\left(\frac{nt}{1-e^{-nt}}-\frac{t}{1-e^{-t}}\right),\] which matches the derivative $u'(t)$.  The constant term is verified by computing $\lim_{t \to 0} u(t)$.

Suppose now $X$ is a random variable with distribution $\prod_{i=1}^k \frac{[a_i]_q}{[b_i]_q}$.  Using~\Cref{eq:ncum}, we compute \begin{align*}\kappa_r(X) &= \kappa_r\left(\prod_{i=1}^k \frac{[a_i]_q}{[b_i]_q}\right) =\sum_{i=1}^k \kappa_r([a_i]_q)-\kappa_r([b_i]_q) = \frac{B_r}{r}\sum_{i=1}^k \left(a_i^r-b_i^r\right). \qedhere\end{align*}
\end{proof}

The remainder of this section is devoted to corollaries of \Cref{prop:cumulant}.

\subsection{Coxeter numerology}

A finite irreducible Coxeter group $W$ of rank $n$ has roots $\Phi$, positive roots $\Phi^+$, degrees $d_1 < \cdots < d_n$, exponents $e_1<\cdots< e_n$, and a Coxeter number $h$.  These satisfy $e_i = d_i-1$, $h=d_n$, and \[\sum_{i=1}^n e_i=\frac{nh}{2}= |\Phi^+|.\]

Following~\cite{burns2012power,suter1998coxeter}, for $W$ a finite Weyl group (a crystallographic Coxeter group), define \[\gamma=\frac{\langle \widetilde{\alpha},\widetilde{\alpha} \rangle}{\langle \widetilde{\alpha}_s,\widetilde{\alpha}_s\rangle}g \chk{g},\] where $\widetilde{\alpha}$ is the highest root, $\widetilde{\alpha}_s$ is the highest short root, $g$ is the dual Coxeter number of $\Phi$, and $\chk{g}$ is the dual Coxeter number of the dual root system $\chk{\Phi}$.  For noncrystallographic types, $\gamma$ is defined in~\Cref{fig:gamma}.  With this definition, R.~Suter found the striking uniform formulas~\cite{suter1998coxeter} \begin{equation}\sum_{i=1}^n e_i^2 = \frac{n(h^2+\gamma-h)}{6} \text{ and } \sum_{i=1}^n e_i^3 = \frac{nh(\gamma-h)}{4}.\label{eq:sut}\end{equation}

\begin{figure}[htbp]
\[\begin{array}{|c|cccccccc|}\hline
\text{Type} & A_n & B_n/C_n & D_n & E_6 & E_7 & E_8 & F_4 & G_2 \\ 
\gamma & (n+1)^2 & 4n^2+2n-2 & (2n-2)^2 & 144 & 324 & 900 & 162 & 48 \\ \hline
\end{array}\]

\[\begin{array}{|c|ccc|}\hline
\text{Type} &H_3 & H_4 & I_2(m)\\
\gamma & 124 & 1116 & 2m^2-5m+6 \\ \hline
\end{array}\]
\caption{Values of $\gamma$ for finite Coxeter groups.  The definition of $\gamma$ is uniform for Weyl groups.}
\label{fig:gamma}
\end{figure}


\subsection{Inversions in Finite Coxeter Groups}
Let $(W,S)$ be a finite irreducible Coxeter system.  The \defn{length} of the shortest word in simple reflections for an element $w \in W$ is written $\ell(w)$.  It is well-known that the generating function for length is
\[\mathrm{Inv}(W;q):=\sum_{w\in W} q^{\ell(w)}=\prod_{i=1}^n [d_i]_q.\]

Applying \Cref{prop:cumulant} gives the following  generalization of~\cite[Theorem 3.1]{kahle2018counting}.

\begin{corollary}
Let $W$ be a finite irreducible Coxeter group, and let $X$ be a random variable with distribution $\mathrm{Inv}(W;q)$.  Then \[\kappa_r(X) = \frac{B_r}{r} \sum_{i=1}^n \left(d_i^r-1\right).\]  In particular, the expectation and variance are \[\mathbb{E}(X)=\frac{|\Phi^+|}{2} \text{ and } \mathbb{V}(X)=\frac{n(\gamma+5h+h^2)}{72}.\]
\end{corollary}

\begin{proof}
The formula for expectation is immediate from the symmetry $w \mapsto w w_\circ$, where $w_\circ$ is the longest element of $W$.  The formula for variance follows from \Cref{eq:sut}.
\end{proof}

\subsection{Rational Catalan Numbers}
Let $W$ be a finite well-generated irreducible complex reflection group with exponents $e_1,\ldots,e_n$ and $p$ coprime to to the Coxeter number $h=d_n$.  The \defn{rational Catalan number} \[\mathrm{Cat}^{[p]}(W;q):=\prod_{i=1}^n \frac{[p+(pe_i \mod h)]_q}{[d_i]_q}\] is the graded character of the representation $eL_{p/h}(\mathsf{triv})$ of the rational Cherednik algebra corresponding to $W$~\cite{gordon2012catalan,stump2015cataland}.  When $W$ is a Coxeter group, multiplication by $p$ simply permutes the exponents modulo $h$.  Applying \Cref{prop:cumulant} gives the following.

\begin{corollary}
Let $W$ be a finite Coxeter group, and let $X$ be a random variable with distribution $\mathrm{Cat}^{[p]}(W;q)$.  Then \[\mathbb{E}(X)=\frac{n(p-1)}{2}, \hspace{2em} \mathbb{V}(X) = \frac{n(p-1)(p+h+1)}{12},\] and for $W$ a crystallographic Coxeter group $\kappa_4(X)=-\frac{n(p-1)(h+p+1)(p^2+ph+\gamma+1)}{120}$.
\label{prop:rat_cat}
\end{corollary}

\begin{proof}
The formulas follow from simple computations and \Cref{eq:sut}.
\end{proof}

Interestingly, the variance in~\Cref{prop:rat_cat} is exactly the expectation computed in~\cite{thiel2017strange} for simply-laced crystallographic Coxeter groups.

\subsection{Minuscule Posets}
Recall that a dominant weight $\lambda$ is called \defn{minuscule} if the weights in its highest weight representation $V_\lambda$ coincides with weights in its Weyl group orbit $\{w \lambda : w \in W\}$.  A \defn{minuscule poset} may then be defined as the order filter of positive roots generated by simple root whose corresponding fundamental weight is minuscule.

A \defn{plane partition of height at most $k$} in a poset $P$ is an order-preserving map $P \to \{0,1,\ldots,k\}$.  It is well-known that the generating function for the number of boxes in plane partitions of height at most $k$ in a minuscule poset has the product formula~\cite{proctor1984bruhat} \[\mathrm{PP}^{[p]}(P;q):=\prod_{p\in P} \frac{[k+\mathrm{ht}(p)]_q}{[\mathrm{ht}(p)]_q}.\]

\begin{corollary}
Let $P$ be a minuscule poset, and let $X$ be a random variable with distribution $\mathrm{PP}^{[p]}(P;q)$.  Then \[\mathbb{E}(X)=\frac{k}{2}|P| \text{ and } \mathbb{V}(X) = \frac{k(k+h)}{12}|P|.\]
\label{prop:min_pp}
\end{corollary}

\begin{proof}
We have that \[\kappa_r(X)=\frac{B_r}{r}\sum_{p \in P} \Big( (k+\mathrm{ht}(p))^r-\mathrm{ht}(p)^r\Big).\]
Expectation is immediate.  For variance, we compute
\begin{align*}
  \frac{1}{12}\sum_{p \in P} \Big( (k+\mathrm{ht}(p))^2-\mathrm{ht}(p)^2 \Big) &= \frac{1}{12}\sum_{p \in P} \left( k^2 + 2k \mathrm{ht}(p) \right)\\
  &=\frac{1}{12}\left(|P|k^2+2k \sum_{p \in P} \mathrm{ht(p)} \right)\\
  &=\frac{1}{12}\left(|P|k^2+|P|hk\right)=\frac{k(k+h)}{12}|P|,
\end{align*}
where $\sum_{p \in P} \mathrm{ht(p)}=\frac{h}{2}|P|$ using the symmetry of $P$.
\end{proof}

Specializing to type $A$ gives the following.

\begin{corollary}
The variance for the number of boxes in plane partitions fitting inside an $a \times b \times c$ box is $\frac{1}{12}abc(a+b+c).$
\end{corollary}

\medskip

A \defn{linear extension} of a poset $P$ is an order-preserving bijection from $P$ to $\{1,2,\ldots,|P|\}$. Relative to a fixed linear extension $\ell$ of a poset, the \defn{major index} of a second linear extension $\ell'$ is the sum of the positions of the descents of $\ell'$---that is, the sum $\sum i,$ where the sum is over all $i$ for which $\ell'\left(\ell^{-1}(i)\right)>\ell'\left(\ell^{-1}(i+1)\right)$.  Recall that the generating function for major index of linear extensions of a minuscule poset is given by \[\mathrm{SYT}(P;q):=\frac{[|P|]!_q}{\prod_{p \in P} [\mathrm{ht}(p)]_q}.\]

\begin{corollary}
Let $P$ be a minuscule poset, and let $X$ be a random variable with distribution $\mathrm{SYT}(P;q)$.  Then $\mathbb{E}(X)=\frac{|P|(|P|+1-h)}{4}$.
\label{prop:min_le}
\end{corollary}
\begin{proof}
Compute using \[\kappa_r(X)=\frac{B_r}{r}\left(\sum_{i=1}^{|P|} i^r-\sum_{p \in P}\mathrm{ht}(p)^r\right).\]
\end{proof}

Specializing to type $A$ gives the following.

\begin{corollary}
The expected value for major index of standard Young tableaux of rectangular shape is given by $\frac{a(a-1)b(b-1)}{4}$.
\end{corollary}

\subsection{Descending Plane Partitions}

The generating function for descending plane partitions by number of boxes is given by \[\mathrm{DPP}(q)=\prod_{i=0}^{n-1} \frac{[3i+1]!_q}{[n+i]!_q}.\]

\begin{corollary}
Let $X$ be a random variable with distribution $\mathrm{DPP}(q)$.  Then \[\mathbb{E}(X)= \frac{1}{6} n(n^2-1) \text{ and } \mathbb{V}(X)= \frac{1}{12} n^2(n^2-1).\]
\label{prop:dpp}
\end{corollary}

\section{Open problems}
\begin{itemize}
\item The problem of determining uniform formulas for higher moments for the expected norm of a weight in a highest weight representation is open.
\item By the Harish-Chandra isomorphism, $\gf$ has $n$ Casimir elements.  Since the Casimir elements live in the center of $U(\gf)$, they generalize the degree two Casimir by acting as a scalar on any highest weight representation $\V(\lambda)$.  It might be interesting to write down explicit formulas.\footnote{After our presentation of this work at FPSAC 2018, we understand that Richard Stanley has made some progress on this problem in type $A$.}  \item The expectations arising from evaluating other natural $W$-symmetric functions besides the norm on the weights of $\V_\lambda$ could be worth looking at.  \item Since Schubert polynomials generalize Schur polynomials---which are equivalent to the Weyl character formula in type $A$---one could ask for expectation of polynomial functions of the exponents in a Schubert polynomial.
\item For $P$ a minuscule poset, it would be interesting to find uniform expressions for $\sum_{p \in P} \mathrm{ht}(p)^r$.
\end{itemize}

\section*{Acknowledgments}
The second author warmly thanks Dennis Stanton for precious help with~\Cref{sec:typea} and Paul Garrett for explaining to him where the Casimir element lives.  An extended abstract of this work was presented at FPSAC 2018~\cite{thielwinnie}.

\bibliographystyle{amsalpha}
\bibliography{winnie_the_pooh}

\end{document}